\documentclass[11pt,letterpaper]{amsart}

\usepackage[T1]{fontenc}
\usepackage{microtype}
\usepackage{amsmath,amssymb,amsthm,mathtools}
\usepackage{enumitem}
\usepackage{aliascnt}
\usepackage{booktabs,tabularx,array}
\usepackage[hidelinks]{hyperref}
\usepackage[nameinlink,capitalise]{cleveref}
\allowdisplaybreaks

\hypersetup{
  pdftitle={Repairing the refined-decoupling proof of the 5/4 planar pinned Falconer theorem},
  pdfauthor={Shalender Singh and Vishnu Priya Singh Parmar},
  pdfsubject={A counterexample to the printed arbitrary-packet formulation and a complete repair},
  pdfkeywords={Falconer distance problem, pinned distance sets, refined decoupling, activity tubes, wave packets, local constancy}
}

\title[Repairing the planar pinned Falconer proof]{Repairing the refined-decoupling proof of the 5/4 planar pinned Falconer theorem}
\author{Shalender Singh}
\address{Independent Researcher, Milpitas, California, USA}
\email{singhshalender@gmail.com}
\author{Vishnu Priya Singh Parmar}
\address{Independent Researcher, Milpitas, California, USA}
\thanks{Both authors are independent researchers in Milpitas, California, USA. Corresponding author: Shalender Singh, \texttt{singhshalender@gmail.com}. ORCID: Shalender Singh, 0000-0003-3677-4043; Vishnu Priya Singh Parmar, 0009-0007-1130-4551.}
\date{August 27, 2026}


\newtheorem{theorem}{Theorem}[section]
\newaliascnt{proposition}{theorem}
\newtheorem{proposition}[proposition]{Proposition}
\aliascntresetthe{proposition}
\newaliascnt{lemma}{theorem}
\newtheorem{lemma}[lemma]{Lemma}
\aliascntresetthe{lemma}
\newaliascnt{corollary}{theorem}
\newtheorem{corollary}[corollary]{Corollary}
\aliascntresetthe{corollary}
\newaliascnt{remark}{theorem}
\newtheorem{remark}[remark]{Remark}
\aliascntresetthe{remark}
\newaliascnt{observation}{theorem}

\aliascntresetthe{observation}

\newcommand{\R}{\mathbb{R}}
\newcommand{\supp}{\operatorname{supp}}
\newcommand{\dist}{\operatorname{dist}}
\newcommand{\RapDec}{\operatorname{RapDec}}
\newcommand{\cN}{\mathcal{N}}
\newcommand{\wh}{\widehat}
\newcommand{\eps}{\varepsilon}
\newcommand{\cW}{\mathcal{W}}

\newcommand{\cQ}{\mathcal{Q}}

\newcommand{\cS}{\mathcal{S}}
\newcommand{\abs}[1]{\lvert#1\rvert}
\newcommand{\norm}[1]{\lVert#1\rVert}

\begin{document}

\begin{abstract}
We show that the literal arbitrary-packet form of the refined-decoupling
estimate printed in the Guth--Iosevich--Ou--Wang proof of the planar pinned
Falconer theorem is false, even after its tube dimensions are normalized.
An explicit collar construction gives a fixed-power counterexample: all
packets are active on one square while none of their smaller labelled tubes
meets that square.  We then give a non-circular repair of the original proof
route.  The analytic input is an enlargement-stable arbitrary-packet theorem
in which a packet is concentrated on an \(a\)-dilate and multiplicity is
counted with a strictly larger \(b\)-dilate.  We prove this theorem directly
from weighted \(\ell^2\) decoupling by an induction that tracks the dilation
margin through parabolic rescaling.

The corrected theorem applies directly to the original Falconer parent
packets after an exact frequency truncation and an absolute small-packet
cutoff; no parent-to-canonical decomposition is needed.  We then rebuild the
good-tube incidence estimate, retain the neighborhood forced by local
constancy, and supply a uniform regularization and limiting argument.  The
principal frequency exponent remains \(-(\alpha+1)/3\), so the energy
argument closes exactly for \(\alpha>5/4\).  The pinned theorem itself is not
contradicted and is also known through later microlocal methods.  A later
canonical wave-packet treatment of refined decoupling overlaps with the
activity-tube viewpoint but not with the counterexample or the repaired
Falconer proof chain.
\end{abstract}

\maketitle

\section{Introduction}

For a compact set \(E\subset\R^2\) and a point \(x\in\R^2\), write
\[
  \Delta_x(E):=\{\abs{x-y}:y\in E\}.
\]
Guth, Iosevich, Ou and Wang proved that
\[
 \dim_{\mathrm H}E>\frac54
 \quad\Longrightarrow\quad
 \abs{\Delta_x(E)}>0\ \text{for some }x\in E
\]
\cite{GIOW}.  Their proof separates a Frostman measure into a good
wave-packet part and a bad train-track part.  The bad part is controlled by
Orponen's radial-projection theorem \cite{Orponen}; the good part uses Liu's
pinned-distance identity \cite{Liu} and a refined-decoupling estimate.

This paper concerns the correctness of that particular proof route.  It does
not claim that the pinned theorem was otherwise unavailable.  Iosevich, Liu
and Xi subsequently proved the corresponding result for two-dimensional
Riemannian manifolds, which supplies a separately published microlocal route
and includes the Euclidean theorem as a local special case \cite{ILX}.
Carbery, Li, Pang and Yung later gave a rigorous canonical Fourier-series
wave-packet treatment of refined decoupling for well-curved curves
\cite{CLPY}.  Their packets are counted on the enlarged tiles on which they
are physically active.  That work overlaps with the canonical
activity-tube viewpoint, but it does not give the collar counterexample
below or repair the incidence, local-constancy and limiting steps of the
original Falconer argument.

The contribution here is therefore specific.  We identify a false literal
arbitrary-packet formulation, give an explicit fixed-power counterexample,
prove a corrected enlargement-stable arbitrary-packet theorem, and use it
to close the original Falconer proof chain.  To the best of our knowledge,
this combination of counterexample and complete correction has not
previously appeared.

\subsection*{Three main results}

The first result isolates the failure independently of the later Falconer
application.

\begin{theorem}[Failure of the printed arbitrary-packet formulation]
\label{thm:false-formulation}
Correct the tube dimensions in the printed refined-decoupling statement to
length \(R\) and radius \(R^{1/2+\delta}\), but retain the convention that a
packet labelled by \(T\) may be essentially supported in \(2T\) while
multiplicity counts cubes meeting \(T\).  Then the asserted estimate is
false for \(p=6\), even with an \(R^\eps\) loss.
\end{theorem}

The proof is the explicit construction in
\cref{prop:intended-counterexample}.  Its two sides have orders
\(R^{1/2}\) and \(R^{1/3+\delta/6+\eps}\), respectively.

For the correction, let \(g\) belong to the normalized phase family of
\cref{app:normalized-family}.  At scale \(R\), let \(I\) range over a
translate of a lattice cover by intervals of length \(R^{-1/2}\), with
centre spacing \(\sigma_0R^{-1/2}\), \(1/2\le\sigma_0\le1\).  Put
\begin{align}
 \theta_I^\Lambda
 &=\{\xi:\xi_1\in32I,\
       \abs{\xi_2-g(\xi_1)}\le4\Lambda R^{-1}\},
 \label{eq:intro-cap}\\
 T_{I,v}
 &=\{X\in[-25R,25R]\times[-R,R]:
       \abs{X_1+g'(c_I)X_2-v}\le\Lambda R^{1/2}\}.
 \label{eq:intro-strip}
\end{align}

\begin{theorem}[Enlargement-stable arbitrary-packet refined decoupling]
\label{thm:enlargement-stable}
Fix \(2\le p\le6\) and \(\epsilon>0\).  There is
\(\vartheta_0(\epsilon)>0\) such that the following holds whenever
\(0<\vartheta\le\vartheta_0(\epsilon)\),
\(R^{\vartheta/2}\le\Lambda\le R^{2\vartheta}\),
\(1\le a\le10\), and \(a+\tfrac12\le b\le20\).
Let \(\cW\) be a finite family of strips \(T=T_{I,v}\), with at most
fourfold overlap for each fixed label \(I\), and let
\[
 F=\sum_{T\in\cW}F_T,\qquad
 \supp\wh F_T\subset\theta_I^\Lambda,\qquad
 \norm{F_T}_{L^p(\R^2\setminus aT)}
 \le R^{-100}\norm{F_T}_p.
\]
Assume that the packet norms are comparable within a factor two.  If \(Y\)
is a union of \(R^{1/2}\)-squares and every such square meets at most \(M\)
of the enlarged strips \(bT\), then
\begin{equation}\label{eq:enlargement-stable}
 \norm{F}_{L^p(Y)}
 \lesssim_{\epsilon,\vartheta,p}
 R^\epsilon
 \left(\frac{M}{\abs{\cW}}\right)^{\frac12-\frac1p}
 \left(\sum_{T\in\cW}\norm{F_T}_p^2\right)^{1/2}.
\end{equation}
\end{theorem}

Appendix~\ref{app:refined-decoupling} proves
\cref{thm:enlargement-stable} from weighted \(\ell^2\) decoupling.  The
supporting and counted strips remain distinct at every induction scale.

The final result repairs the original pinned-distance route.

\begin{theorem}[Repair of the planar pinned Falconer proof]
\label{thm:main}
Let \(E\subset\R^2\) be compact and let
\(5/4<\alpha<\dim_{\mathrm H}E\).  With the corrected activity and goodness
conventions below, for every \(\eps>0\) and \(N>0\),
\begin{equation}\label{eq:main-frequency}
\begin{aligned}
 \int_{E_2}\abs{\mu_{1,\mathrm{good}}*\wh{\sigma_r}(x)}^2\,d\mu_2(x)
 &\lesssim_{\eps,\alpha,R_0}
 r^{-\frac{\alpha+1}{3}+\eps}\,r^{-1}
 \int_{\R^2}\abs{\wh\mu_1(\xi)}^2\psi_r(\xi)\,d\xi\\
 &\quad+C_Nr^{-N}.
\end{aligned}
\end{equation}
Here \(\sigma_r\) is normalized arc length on \(S_r^1\), and \(\psi_r\)
is comparable to one on the unit neighborhood of \(S_r^1\) and rapidly
decreasing away from it.  Consequently there exists \(x\in E\) such that
\(\abs{\Delta_x(E)}>0\).
\end{theorem}

After integration in \(r\), the energy order is
\[
 \gamma=2-\frac{\alpha+1}{3}+\eps.
\]
Thus \(\gamma<\alpha\) precisely for \(\alpha>5/4\), after choosing
\(\eps\) sufficiently small.

\subsection*{Why no parent-to-canonical reduction is needed}

A canonical wave-packet decomposition is one possible way to repair the
support convention, but it is not necessary once
\cref{thm:enlargement-stable} is available.  The original Falconer parent
packet already has the required physical geometry.  Two technical
hypotheses remain: the spatial cutoff destroys exact Fourier support, and
an absolute \(\RapDec(r)\) tail is not automatically small relative to a
packet of arbitrarily small norm.  Section~\ref{sec:direct-parents} resolves
these points by an exact frequency truncation and an \emph{absolute}
small-packet cutoff.  The corrected arbitrary-packet theorem then applies
directly to the parent family.  This eliminates the longer common-frame
parent-to-canonical construction and makes the correction both shorter and
stronger.

\subsection*{One-sided localization and parameter order}

We use the following standard preliminary localization.

\begin{lemma}[One-sided localization of the Frostman pieces]
\label{lem:one-sided-localization}
Let \(\mu\) be an \(\alpha\)-Frostman probability measure on a compact set
\(E\subset\R^2\).  There are disjoint balls
\(B_i=B(z_i,\rho_*)\) of positive \(\mu\)-measure such that, after restricting
and normalizing,
\[
 \mu_i=\frac{\mu|_{B_i}}{\mu(B_i)},\qquad
 E_i=\supp\mu_i,\quad i=1,2,
\]
the measures remain \(\alpha\)-Frostman and all directions
\[
 \frac{x-y}{\abs{x-y}},\qquad y\in E_1,\ x\in E_2,
\]
lie in one fixed arbitrarily short arc of \(S^1\).
\end{lemma}

\begin{proof}
Choose distinct \(z_1,z_2\in\supp\mu\), put
\(d=\abs{z_2-z_1}>0\), and take \(0<\rho_*\ll d\).  The normalized
restrictions remain Frostman, with constants increased only by their fixed
masses.  For \(y\in B_1\), \(x\in B_2\), the vector \(x-y\) differs from
\(z_2-z_1\) by at most \(2\rho_*\), so the normalized directions lie in an
arc of length \(O(\rho_*/d)\).  A similarity gives the usual bounded
normalization.
\end{proof}

We henceforth use these localized pieces.  A positive-measure pinned
distance set for \(E_1\) is one for the original \(E\).

Fix the final loss \(\eps_0>0\).  Orponen's theorem first determines
\(p=p(\alpha)>1\), and then \(c_\alpha\) is chosen in the bad-pair estimate.
Next choose \(\delta>0\) so small that
\[
 c_\alpha\delta<\frac12,\qquad
 C_\alpha\delta<\frac{\eps_0}{100}
\]
for every loss below.  The rapid-decay order and the absolute small-packet
cutoff are chosen next, and \(R_0\) last.

Section~2 gives the diagnosis and counterexample.  Section~3 derives the
Falconer-scale form of \cref{thm:enlargement-stable}.  Sections~4--7 apply it
directly to the parent packets and prove \cref{eq:main-frequency}.
Section~8 supplies the uniform limiting argument.  Appendix~A proves the
corrected refined-decoupling theorem, and Appendix~B records additional
corrections to the published article.

\section{The printed gaps and the collar counterexample}
\label{sec:gaps}

\subsection{Incompatible tube dimensions}

The setup of \cite[Theorem~4.2]{GIOW} uses tubes of length comparable to
\(R^{1+\delta}\) and then requires every selected tube to lie in \(B_R\).
The proof later treats the tubes as having length \(R\) and radius
\(R^{1/2+\delta}\), and the physical rescaling produces unit-length tubes
of width \(R^{-1/2+\delta}\).  The dimensionally consistent normalization is
therefore
\[
 \text{length }R,\qquad \text{radius }R^{1/2+\delta},
\]
with tails outside a fixed enlargement treated as negligible.

\subsection{Labelled tubes versus activity tubes}

The same setup permits a packet \(f_T\) to be essentially supported in
\(2T\), whereas the multiplicity of a spatial square is the number of
smaller labelled tubes \(T\) meeting it.  The latter count does not control
the number of packets active in \(2T\setminus T\).

\begin{proposition}[Collar counterexample]
\label{prop:intended-counterexample}
Fix \(0<\eps<1/12\) and \(0<\delta<\eps/100\), and take \(p=6\).
Replace the inconsistent printed dimensions by tubes of length \(R\) and
radius \(R^{1/2+\delta}\), but retain the support and multiplicity
conventions just described.  Then the asserted estimate with loss \(R^\eps\)
fails for all sufficiently large \(R\).
\end{proposition}

\begin{proof}
Let \(N=\lfloor cR^{1/2}\rfloor\), with \(c>0\) small, and choose \(N\)
distinct \(R^{-1/2}\)-caps on \(S^1\), with centers \(\omega_j\).  Let \(Q\)
be the \(R^{1/2}\)-square centered at the origin and put
\(w=R^{1/2+\delta}\).

For each \(j\), choose a length-\(R\) tube \(T_j\) in direction \(\omega_j\)
whose axis is at transverse distance \(3w/2\) from the origin.  Let \(U_j\)
be the tube in the same direction whose axis passes through the origin and
whose radius is \(w/4\).  For large \(R\),
\[
 Q\subset\tfrac12U_j,\qquad
 U_j\subset2T_j\setminus T_j,\qquad
 Q\cap T_j=\varnothing.
\]
Choose a smooth bump \(\chi_j\) equal to one on \(Q\), supported in \(U_j\),
and adapted to its dimensions, and set
\[
 f_j(x)=a\chi_j(x)e^{2\pi i\omega_j\cdot x},\qquad
 f=\sum_{j=1}^N f_j.
\]
Then \(f_j\) is supported in \(2T_j\), and its Fourier transform is
concentrated in the required cap.  On the fixed disc
\(\abs{x}\le1/100\), all phases lie in a common sector, so
\[
 \norm{f}_{L^6(Q)}\gtrsim \abs{a}\,N
 \gtrsim\abs{a}\,R^{1/2}.
\]
Every labelled tube misses \(Q\), so the printed hypothesis permits \(M=1\).
Moreover,
\[
 \norm{f_j}_6\lesssim\abs{a}\,\abs{U_j}^{1/6}
 \lesssim\abs{a}\,R^{1/4+\delta/6}.
\]
The claimed right-hand side is therefore at most
\[
 C_\eps R^\eps N^{-1/3}
 \left(N\abs{a}^2R^{1/2+\delta/3}\right)^{1/2}
 \lesssim C_\eps\abs{a}\,R^{1/3+\delta/6+\eps},
\]
which is smaller than the preceding lower bound because
\(\delta/6+\eps<1/6\).
\end{proof}

\begin{remark}
This proposition does not contradict a canonical refined-decoupling theorem:
in a canonical decomposition the activity tube is part of the packet data and
that enlarged tube is what the multiplicity counts.  The counterexample
deliberately separates the label from the region carrying the packet.
\end{remark}

\subsection{Local constancy does not retain the same set}

If \(\supp\wh H\subset B(0,Cr)\), choose \(\varphi_r\) with
\(H=H*\varphi_r\).  Cauchy--Schwarz and Fubini yield
\[
 \int_Y\abs{H}^2\,d\mu
 \lesssim
 \int\abs{H(z)}^2
       \bigl((\mu|_Y)*\abs{\varphi_r}\bigr)(z)\,dz.
\]
The density on the right is concentrated on a neighborhood of \(Y\), not
on \(Y\) itself.  The correction in Section~\ref{sec:incidence-local}
retains that neighborhood and transfers the activity multiplicity to its
covering squares.

\begin{table}[!htbp]
\caption{Original proof locations, literal issues, and repairs.}
\label{tab:comparison}
\small
\renewcommand{\arraystretch}{1.16}
\begin{tabularx}{\textwidth}{>{\raggedright\arraybackslash}p{0.24\textwidth}
                              >{\raggedright\arraybackslash}p{0.33\textwidth}
                              >{\raggedright\arraybackslash}X}
\toprule
Original location & Issue or proof obligation & Counterexample or repair\\
\midrule
GIOW Theorem 4.2
& Packets may live in \(2T\), while multiplicity counts \(T\)
& \Cref{prop:intended-counterexample}; count a strictly larger activity tube\\
GIOW Theorem 4.2
& Tube dimensions conflict with \(T\subset B_R\) and with the rescaling
& Use length \(R\), width \(\Lambda R^{1/2}\), and track the dilation margin\\
GIOW Proposition 5.3
& Parent packets do not initially have exact Fourier support or relative tails
& Exact frequency truncation plus an absolute small-packet cutoff; apply
  \cref{thm:enlargement-stable} directly\\
Local-constancy step
& The mollified density is integrated over the same set
& Retain the neighborhood \(Y^+\) and cover it by nearby squares\\
Definition of the good part
& A rigorous approximate-identity limit is deferred
& Uniform smoothing, pin selection, and compact limiting argument\\
GIOW Lemma 7.5
& The displayed \(L^2\) remainder is unsquared
& Replace it by \(O((A+B)^2)\)\\
Infinite train-track paragraph
& Nesting and persistence of lower bounds are not proved
& Withdraw the assertion unless an explicit Moran construction is supplied\\
\bottomrule
\end{tabularx}
\end{table}

\section{The corrected theorem at the Falconer scale}
\label{sec:corrected-decoupling}

Theorem~\ref{thm:enlargement-stable} is proved in
Appendix~\ref{app:refined-decoupling}.  Demeter's formulation
\cite[Theorem~1.4]{DemeterRefined} is a later restatement of the canonical
GIOW result.  Carbery--Li--Pang--Yung give a later full canonical treatment
for well-curved curves \cite{CLPY}.  Neither is used to prove the
arbitrary-packet theorem here.

We now record the direct unit-scale corollary used for the Falconer parents.

\begin{corollary}[Falconer-scale arbitrary-packet form]
\label{cor:falconer-packets}
Fix \(2\le p\le6\), \(\eps>0\), and \(0<\delta\le\delta_0(\eps)\).
For all sufficiently large \(r\), let \(\cW\) be a finite family of
unit-length tubes \(T\subset B(0,3)\) with one common half-width \(\rho\)
satisfying
\[
 c_0r^{-1/2+\delta}\le\rho\le C_0r^{-1/2+\delta}.
\]
The directions are drawn from the fixed chart-adapted \(r^{-1/2}\)-lattices
described in the proof.  For \(c>0\), the notation \(cT\) means the
concentric transverse dilation about the full axis of \(T\), restricted to
\(B(0,3)\).  For each direction, assume the core tubes have
bounded overlap.  Associate to \(T\) an arc \(\theta(T)\subset S_r^1\) of
tangential length \(O(r^{1/2})\), normal to \(T\).  Suppose
\[
\begin{aligned}
 f&=\sum_{T\in\cW}f_T,\\
 \supp\wh f_T&\subset
 \cN_{r^{3\delta/4}}\bigl(\theta(T)\bigr),\\
 \norm{f_T}_{L^p(\R^2\setminus2.1T)}
 &\le r^{-100}\norm{f_T}_p.
\end{aligned}
\]
and that the packet norms are comparable within a factor two.  Let \(Y\)
be a union of \(r^{-1/2}\)-squares such that
\[
 \#\{T\in\cW:3T\cap60q\ne\varnothing\}\le M
 \qquad(q\subset Y).
\]
Then
\begin{equation}\label{eq:falconer-packet-cor}
 \norm{f}_{L^p(Y)}
 \lesssim_{\eps,\delta,p}
 r^\eps
 \left(\frac{M}{\abs{\cW}}\right)^{\frac12-\frac1p}
 \left(\sum_{T\in\cW}\norm{f_T}_p^2\right)^{1/2}.
\end{equation}
\end{corollary}

\begin{proof}
Cover the circle by a fixed number of arcs.  On one arc rotate so that its
center is \((0,-1)\), write
\[
 g_0(t)=1-\sqrt{1-t^2},
 \qquad \lambda=\frac1{12},\qquad \kappa=1.92,
\]
and use the frequency normalization
\[
 L(\xi)=\left(\frac{\xi_1}{r\lambda},
              \frac{\kappa(\xi_2+r)}{r\lambda^2}\right).
\]
The normalized phase
\[
 g(u)=\frac{\kappa}{\lambda^2}g_0(\lambda u)
\]
satisfies \(1.92\le g''\le2.01\) on the chart.  Put
\[
 R=\frac{3r\lambda^2}{\kappa}.
\]
Choose the annular angular partition, within this fixed chart, so that its
normalized tangential labels are a translate of a fixed
\(R^{-1/2}\)-lattice.  At chart boundaries a fixed partition of unity and a
finite coloring preserve this property and the bounded overlap by label.

The dual physical map is
\[
 P(x)=\left(r\lambda x_1,\frac{r\lambda^2}{\kappa}x_2\right).
\]
The image of each tube is a strip of the form
\eqref{eq:intro-strip} with
\[
 \Lambda=\frac{r\lambda\rho}{R^{1/2}}
 \asymp r^\delta.
\]
For large \(r\),
\(R^{\delta/2}\le\Lambda\le R^{2\delta}\).
The frequency neighborhood in the hypothesis maps into
\(\theta_I^\Lambda\): its tangential size is \(O(R^{-1/2})\), and its
normal thickness is
\(O(r^{-1+3\delta/4})=o(\Lambda R^{-1})\).
The fixed linear normalization also gives
\[
 P(2.1T)\subset2.2T_{I,v},\qquad
 2.8T_{I,v}\cap P(B(0,3))\subset P(3T).
\]
A normalized \(R^{1/2}\)-square pulls back to a bounded-eccentricity
rectangle of dimensions \(O(r^{-1/2})\), hence is contained in \(60q\) for
one original grid square \(q\).  Thus the normalized multiplicity is at most
an absolute multiple of \(M\).  After a finite coloring, the overlap for one
label is at most four.  Apply \cref{thm:enlargement-stable} with
\(a=2.2\), \(b=2.8\), and absorb the fixed number of charts, colors and
covering squares into the implicit constant.  The Jacobian of \(P\) occurs
on both sides of the inequality.  A family with finitely many comparable
widths is handled by splitting it into those width classes.
\end{proof}

\begin{remark}[Relation to the later canonical literature]
Carbery--Li--Pang--Yung construct canonical packets by an exact
Fourier-series expansion and prove refined decoupling by counting enlarged
tiles on which those packets are active \cite{CLPY}.  Corollary
\ref{cor:falconer-packets} is different in scope: it permits arbitrary
packets satisfying exact cap support and a relative physical-tail condition.
This is the form that makes a direct application to the Falconer parent
packets possible.
\end{remark}

\section{The enlarged goodness region and the bad part}
\label{sec:bad}

Throughout this section \(E_1,E_2\) are the localized pieces from
\cref{lem:one-sided-localization}.  Let \(\ell_T\) be the axis of a parent
tube of half-width \(\rho_j\).  Define the infinite concentric activity and
goodness strips by
\[
 T^{\mathrm{act}}
 =\{x:\dist(x,\ell_T)\le3\rho_j\},\qquad
 T^{\mathrm{good}}
 =\{x:\dist(x,\ell_T)\le A_{\mathrm{good}}\rho_j\},
\]
where \(A_{\mathrm{good}}>3\) is a fixed absolute constant chosen so large
that
\begin{equation}\label{eq:Agood-choice}
 (A_{\mathrm{good}}-3)c_\rho>300
\end{equation}
for the fixed lower width constant
\(\rho_j\ge c_\rho r^{-1/2+\delta}\).
We call \(T\in\mathbb T_{j,\tau}\) bad if
\begin{equation}\label{eq:bad-definition}
 \mu_2(T^{\mathrm{good}})
 \ge R_j^{-1/2+c_\alpha\delta},
\end{equation}
and good otherwise.

\begin{lemma}[Bad-pair estimate for the enlarged strips]
\label{lem:bad-pair}
Assume \(\alpha>1\).  There is \(c_\alpha>1\), depending only on the
radial-projection exponent for \(\alpha\), such that
\begin{equation}\label{eq:bad-pair}
 (\mu_1\times\mu_2)(\operatorname{Bad}_j)
 \lesssim R_j^{-2\delta},
\end{equation}
where
\[
 \operatorname{Bad}_j
 =\{(y,x):\text{some bad }T\in\mathbb T_j
                 \text{ satisfies }x,y\in2T\}.
\]
\end{lemma}

\begin{proof}
Fix \(y\in E_1\).  If \(y\in2T\) and \(T\) is bad, the one-sided
localization places
\[
 P_y(T^{\mathrm{good}}\cap E_2)
 \subset A_y(T)
\]
in one arc of length \(O(R_j^{-1/2+\delta})\).  Since \(\mu_2\) is supported
on \(E_2\),
\[
 P_y\mu_2(A_y(T))
 \ge\mu_2(T^{\mathrm{good}})
 \ge R_j^{-1/2+c_\alpha\delta}.
\]
A Vitali subfamily shows
\[
 \abs{P_y(\operatorname{Bad}_j(y))}
 \lesssim R_j^{-(c_\alpha-1)\delta}.
\]
Orponen's theorem gives \(p=p(\alpha)>1\) such that
\[
 \int\norm{P_y\mu_2}_{L^p(S^1)}^p\,d\mu_1(y)<\infty.
\]
H\"older in the arc and then in \(y\) yields
\[
 (\mu_1\times\mu_2)(\operatorname{Bad}_j)
 \lesssim
 R_j^{-(c_\alpha-1)(1-1/p)\delta}.
\]
Choose \(c_\alpha\) so that
\((c_\alpha-1)(1-1/p)\ge2\).
\end{proof}

\begin{proposition}[Uniform control of the enlarged bad part]
\label{prop:bad-L1}
For \(R_0\) sufficiently large there is \(E_2'\subset E_2\) with
\(\mu_2(E_2')\ge1-10^{-3}\) such that
\[
 \norm{d_x{}_*\mu_1-d_x{}_*\mu_{1,\mathrm{good}}}_{L^1}
 \le10^{-3}\qquad(x\in E_2').
\]
The constants remain uniform for smooth approximations with the same
Frostman and separation constants.
\end{proposition}

\begin{proof}
The stationary-phase and \(L^1\) estimates of
\cite[Lemmas~3.1--3.4]{GIOW} do not use the definition of badness.  They give
\begin{equation}\label{eq:bad-L1-fiber}
 \norm{d_x{}_*\mu_1-d_x{}_*\mu_{1,\mathrm{good}}}_{1}
 \lesssim
 \sum_{j\ge1}R_j^\delta\mu_1(\operatorname{Bad}_j(x))
 +\RapDec(R_0).
\end{equation}
For fixed separated \(x,y\), the admissible directions occupy an
\(O(R_j^{-1/2+\delta})\) window in an \(R_j^{-1/2}\)-net, which gives the
factor \(R_j^\delta\).

By Fubini and \cref{lem:bad-pair},
\[
 \int\mu_1(\operatorname{Bad}_j(x))\,d\mu_2(x)
 \lesssim R_j^{-2\delta}.
\]
Let
\(B_j=\{x:\mu_1(\operatorname{Bad}_j(x))>R_j^{-3\delta/2}\}\).
Then \(\mu_2(B_j)\lesssim R_j^{-\delta/2}\).  For large \(R_0\), the union
of the \(B_j\) has measure at most \(10^{-3}\), and outside it
\eqref{eq:bad-L1-fiber} is bounded by
\(\sum_jR_j^{-\delta/2}+\RapDec(R_0)\le10^{-3}\).
\end{proof}

\section{Direct preparation of the Falconer parent packets}
\label{sec:direct-parents}

This section answers the structural question raised by the corrected
arbitrary-packet theorem: it can be applied directly to the original parent
packets.  A canonical parent-to-core decomposition is unnecessary.

Fix \(r>10R_0\), a scale \(R_j\sim r\), and a good parent
\(T\in\mathbb T_{j,\tau}\).  Let
\(\eta_1\in C_c^\infty(B(0,3/2))\) equal one on the unit disc and define
\begin{equation}\label{eq:parent-fT}
 f_T=\eta_1\bigl(M_T\mu_1*\wh{\sigma_r}\bigr).
\end{equation}
The stationary-phase calculation in \cite[p.~33]{GIOW} gives
\[
 f_T=\RapDec(r)\quad\text{off }2T,
\]
while \(\wh{M_T\mu_1}\) is rapidly decreasing off \(2\tau\).
The spatial cutoff \(\eta_1\) prevents \(\wh f_T\) from having exact compact
support, so we make one exact truncation.

\begin{lemma}[Exact parent-packet truncation]
\label{lem:parent-truncation}
There are smooth frequency cutoffs \(\chi_T\) and functions
\[
 \widetilde f_T=(\chi_T\wh f_T)^\vee
\]
such that, for every \(N\),
\begin{align}
 \supp\wh{\widetilde f_T}
 &\subset\cN_{r^{3\delta/4}}\bigl(\theta(T)\bigr),
 \label{eq:parent-exact-support}\\
 \norm{f_T-\widetilde f_T}_6
 &\le C_Nr^{-N},
 \label{eq:parent-trunc-error}\\
 \norm{\widetilde f_T}_{L^6(\R^2\setminus2.1T)}
 +\norm{\widetilde f_T}_{L^6(\R^2\setminus B(0,3))}
 &\le C_Nr^{-N}.
 \label{eq:parent-absolute-tail}
\end{align}
Moreover,
\begin{equation}\label{eq:parent-L2-crude}
 \norm{f_T}_2+\norm{\widetilde f_T}_2\lesssim1,
\end{equation}
and \(\norm{\check\chi_T}_1\lesssim1\), uniformly in \(T\) and \(r\).
\end{lemma}

\begin{proof}
The identity
\[
 \wh f_T=\wh{\eta_1}*
 \bigl(\wh{M_T\mu_1}\,\sigma_r\bigr)
\]
and the Schwartz decay of \(\wh{\eta_1}\) show that \(\wh f_T\) is rapidly
decreasing outside the \(r^{\delta/2}\)-neighborhood of
\(2\tau\cap S_r^1\).  Choose \(\chi_T=1\) on that neighborhood and supported
in the \(r^{3\delta/4}\)-neighborhood of the slightly larger arc
\(\theta(T)\).  It may be chosen in an anisotropic box of tangential size
\(O(r^{1/2})\) and normal size \(O(r^{3\delta/4})\).  Hausdorff--Young gives
\eqref{eq:parent-trunc-error}, and the support assertion is exact.

For later control of absolute errors, note first that
\(\norm{M_T\mu_1}_1\lesssim1\): indeed,
\(M_T\mu_1=\eta_T(\check\psi_{j,\tau}*\mu_1)\), and the inverse
multiplier kernels have uniformly bounded \(L^1\)-norm.  Hence
\(\norm{\widehat{M_T\mu_1}}_\infty\lesssim1\), so the normalized-circle
extension in \eqref{eq:parent-fT}, together with the compact cutoff
\(\eta_1\), gives \(\norm{f_T}_2\lesssim1\).  Young's inequality and the
uniform \(L^1\)-bound for \(\check\chi_T\) then give
\eqref{eq:parent-L2-crude}.

The inverse transform satisfies
\begin{equation}\label{eq:chi-kernel}
 \abs{\check\chi_T(x)}
 \le C_N r^{1/2+3\delta/4}
 \left(1+r^{1/2}\abs{x\cdot e_T^\perp}
          +r^{3\delta/4}\abs{x\cdot e_T}\right)^{-N},
\end{equation}
and has uniformly bounded \(L^1\)-norm.  Convolving with \(f_T\), which is
supported in \(B(0,3/2)\) and rapidly decreasing off \(2T\), proves
\eqref{eq:parent-absolute-tail}.  Indeed, outside \(2.1T\) the transverse
separation from \(2T\) is a fixed multiple of
\(r^{-1/2+\delta}\), so \eqref{eq:chi-kernel} gains arbitrary powers of
\(r^{-\delta}\); outside \(B(0,3)\), either the longitudinal or transverse
separation from \(B(0,3/2)\) is fixed.
\end{proof}

The preceding tails are absolute.  To obtain the relative tail required by
\cref{cor:falconer-packets}, we use an absolute small-packet cutoff.  At one
frequency scale there are at most \(Cr^{C_0}\) parent packets.  Fix \(K\)
large and discard those satisfying
\begin{equation}\label{eq:absolute-small}
 \norm{\widetilde f_T}_6<r^{-K}.
\end{equation}
For every retained packet, choose the decay order in
\eqref{eq:parent-absolute-tail} larger than \(K+200\).  Then
\begin{equation}\label{eq:parent-relative-tail}
 \norm{\widetilde f_T}_{L^6(\R^2\setminus2.1T)}
 \le r^{-200}\norm{\widetilde f_T}_6.
\end{equation}
The discarded family is shown to be negligible after
\cref{lem:local-constancy}; see \cref{lem:small-parents}.  The retained norms
range between \(r^{-K}\) and a fixed power of \(r\), so they split into
\(O(\log r)\) dyadic classes.  Each class now satisfies every hypothesis of
\cref{cor:falconer-packets}.

\section{Incidence and local constancy}
\label{sec:incidence-local}

Let \(\cW_\lambda\) be one dyadic \(L^6\)-class of retained good parent
packets from \cref{sec:direct-parents}, and put \(W=\abs{\cW_\lambda}\).
Let \(s=r^{-1/2}\) and let \(\cQ_r\) be the grid of \(s\)-squares meeting
\(B(0,2)\).  Define
\begin{equation}\label{eq:parent-multiplicity}
 m(q)=\#\{T\in\cW_\lambda:100q\cap T^{\mathrm{act}}\ne\varnothing\},
 \qquad T^{\mathrm{act}}=3T.
\end{equation}
For dyadic \(M\ge1\), set
\[
 \cQ_M=\{q\in\cQ_r:M\le m(q)<2M\},
 \qquad
 Y_M=\bigcup_{q\in\cQ_M}q.
\]

\begin{lemma}[Corrected parent incidence estimate]
\label{lem:incidence}
For all sufficiently large \(r\),
\begin{equation}\label{eq:incidence}
 \mu_2\bigl(\cN_{5s}(Y_M)\bigr)
 \lesssim\frac{W}{M}\,r^{-1/2+c_\alpha\delta}.
\end{equation}
\end{lemma}

\begin{proof}
For each grid square put
\[
 \cW(q)=\{T\in\cW_\lambda:100q\cap3T\ne\varnothing\}.
\]
Define
\[
 J=\sum_{q\in\cQ_M}\sum_{T\in\cW(q)}\mu_2(200q).
\]
Each \(q\in\cQ_M\) has at least \(M\) incidences.  Also, \(200q\)
contains the \(5s\)-neighborhood of \(q\).  Hence
\[
 J\ge M\mu_2(\cN_{5s}(Y_M)).
\]
For a fixed incident \(T,q\), choose \(z\in100q\cap3T\).  If \(y\in200q\),
then \(\abs{y-z}\le213s\).  Since
\(\rho_j\ge c_\rho sr^\delta\), \eqref{eq:Agood-choice} gives
\[
 \dist(y,\ell_T)\le3\rho_j+213s
 <A_{\mathrm{good}}\rho_j
\]
for large \(r\).  Thus \(200q\subset T^{\mathrm{good}}\).  The grid family
\(\{200q\}\) has bounded overlap, so goodness gives
\[
 \sum_{q:\,100q\cap3T\ne\varnothing}\mu_2(200q)
 \lesssim\mu_2(T^{\mathrm{good}})
 \lesssim r^{-1/2+c_\alpha\delta}.
\]
Sum over \(T\) and compare with the lower bound.
\end{proof}

For large \(N\), let
\[
 \Phi_r(x)=C_Nr^2(1+r\abs{x})^{-N},\qquad \int\Phi_r=1.
\]

\begin{lemma}[Localized convolution with set enlargement]
\label{lem:local-constancy}
Let \(H\in L^2(\R^2)\) satisfy
\(\supp\wh H\subset B(0,Cr)\), let \(\mu\) be a probability measure, and let
\(Y\) be a union of \(s\)-squares.  Put
\[
 Y^+=\cN_{r^{-3/4}}(Y).
\]
For every \(K>0\), if \(N\) is sufficiently large, then
\begin{equation}\label{eq:local-convolution}
 \int_Y\abs{H}^2\,d\mu
 \lesssim_C
 \int_{Y^+}\abs{H(x)}^2(\mu*\Phi_r)(x)\,dx
 +r^{-K}\norm{H}_2^2.
\end{equation}
If \(\mu(B(x,t))\le C_\mu t^\alpha\), then
\begin{align}
 \norm{\mu*\Phi_r}_\infty
 &\lesssim r^{2-\alpha},
 \label{eq:mollifier-Linfty}\\
 \int_{Y^+}\mu*\Phi_r
 &\le\mu\bigl(\cN_{2s}(Y)\bigr)+O_K(r^{-K}).
 \label{eq:mollifier-mass}
\end{align}
\end{lemma}

\begin{proof}
Choose a Schwartz function \(\varphi\) with
\(\wh\varphi=1\) on \(B(0,2C)\) and put
\(\varphi_r(x)=r^2\varphi(rx)\).  Then \(H=H*\varphi_r\).
Cauchy--Schwarz, integration against \(\mu|_Y\), and Fubini give
\[
 \int_Y\abs{H}^2\,d\mu
 \lesssim
 \int\abs{H(z)}^2
 ((\mu|_Y)*\abs{\varphi_r})(z)\,dz.
\]
Outside \(Y^+\) the kernel is \(O_K(r^{-K})\); on \(Y^+\) it is bounded by
a multiple of \(\mu*\Phi_r\).  The \(L^\infty\) estimate follows by dyadic
annuli.  For \eqref{eq:mollifier-mass}, use Fubini and the rapid tail of
\(\Phi_r\).
\end{proof}

Let \(Y_M^{\mathrm{cov}}\) be the union of all grid squares \(q'\) meeting
\(Y_M^+\).  Then, for some \(q\in\cQ_M\),
\begin{equation}\label{eq:cover-four}
 q'\subset4q,\qquad 60q'\subset100q.
\end{equation}
Consequently
\begin{equation}\label{eq:direct-multiplicity}
 \#\{T\in\cW_\lambda:3T\cap60q'\ne\varnothing\}
 \le m(q)<2M.
\end{equation}
Thus \cref{cor:falconer-packets} applies on \(Y_M^{\mathrm{cov}}\) with
multiplicity \(2M\).

\begin{lemma}[Absolute small parent packets are negligible]
\label{lem:small-parents}
Let \(\cS\) contain at most \(Cr^{C_0}\) functions with Fourier support in
\(B(0,2r)\), \(L^6\)-norm below \(r^{-K}\), and
\(L^2\)-norm at most \(r^{C_1}\).  If \(K\) is chosen after a desired
\(N\), then
\[
 \int_{B(0,2)}
 \left|\sum_{T\in\cS}\widetilde f_T\right|^2\,d\mu_2
 =O_N(r^{-N}).
\]
\end{lemma}

\begin{proof}
Minkowski gives
\[
 \left\|\sum_{T\in\cS}\widetilde f_T\right\|_6
 \lesssim r^{C_0-K}.
\]
Apply \cref{lem:local-constancy} on a fixed covering of \(B(0,2)\).
Using \eqref{eq:mollifier-Linfty}, total mass one, and H\"older,
\[
 \int_{B(0,2)}|H|^2\,d\mu_2
 \lesssim r^{(2-\alpha)/3}\norm{H}_6^2
 +r^{-K_1}\norm{H}_2^2.
\]
Here \(\norm{H}_2\lesssim r^{C_0+C_1}\).  Choose the local-constancy
decay order \(K_1>2(C_0+C_1)+N\), and then choose the absolute packet
cutoff \(K>C_0+(2-\alpha)/6+N/2\).  Both terms are \(O_N(r^{-N})\).
\end{proof}

\section{The corrected principal estimate}
\label{sec:principal}

For \(r>10R_0\), convolution with \(\wh{\sigma_r}\) kills the
low-frequency term, and only \(O(1)\) scales \(R_j\sim r\) remain.
On \(B(0,1)\),
\[
 \mu_{1,\mathrm{good}}*\wh{\sigma_r}
 =\sum_{R_j\sim r}\sum_\tau
   \sum_{\substack{T\in\mathbb T_{j,\tau}\\T\ \mathrm{good}}}
   f_T+\RapDec(r).
\]
Replace \(f_T\) by \(\widetilde f_T\) using
\eqref{eq:parent-trunc-error}, discard the absolute small class by
\cref{lem:small-parents}, split into the \(O(1)\) scale families
\(R_j\sim r\), and dyadically pigeonhole the retained
\(\norm{\widetilde f_T}_6\).  Fix one scale and one norm class
\(\cW_\lambda\), write
\[
 \widetilde f_\lambda
 =\sum_{T\in\cW_\lambda}\widetilde f_T,
 \qquad W=\abs{\cW_\lambda},
\]
and use the multiplicity classes of \cref{sec:incidence-local}.

Put
\[
 A_M=\mu_2(\cN_{5s}(Y_M)).
\]
By \cref{lem:local-constancy}, H\"older with exponents \(3\) and \(3/2\),
and \eqref{eq:mollifier-Linfty}--\eqref{eq:mollifier-mass},
\begin{equation}\label{eq:holder-local}
 \int_{Y_M}\abs{\widetilde f_\lambda}^2\,d\mu_2
 \lesssim
 \norm{\widetilde f_\lambda}_{L^6(Y_M^{\mathrm{cov}})}^2
 r^{(2-\alpha)/3}(A_M+r^{-K})^{2/3}
 +O_N(r^{-N}).
\end{equation}
The \(r^{-K}\) term is rapidly decreasing after increasing \(K\), because
the remaining factors are polynomial in \(r\).

Apply \cref{cor:falconer-packets} at \(p=6\).  Exact Fourier support is
\eqref{eq:parent-exact-support}, relative physical concentration is
\eqref{eq:parent-relative-tail}, and the covering-square multiplicity is
\eqref{eq:direct-multiplicity}.  Hence
\begin{equation}\label{eq:direct-refined}
 \norm{\widetilde f_\lambda}_{L^6(Y_M^{\mathrm{cov}})}^2
 \lesssim r^\eps
 \left(\frac{M}{W}\right)^{2/3}
 \sum_{T\in\cW_\lambda}\norm{\widetilde f_T}_6^2.
\end{equation}
By \cref{lem:incidence},
\begin{equation}\label{eq:cancellation-factor}
 \left(\frac{M}{W}\right)^{2/3}A_M^{2/3}
 \lesssim r^{-1/3+\frac23c_\alpha\delta}.
\end{equation}

It remains to estimate one parent packet.  Since
\(\widetilde f_T\) is rapidly decreasing off \(2.1T\), whose area in the
observation ball is \(O(r^{-1/2+\delta})\), and
\(\norm{\check\chi_T}_1\lesssim1\),
\[
 \norm{\widetilde f_T}_6
 \lesssim r^{-1/12+\delta/6}\norm{f_T}_\infty+\RapDec(r).
\]
Writing the circle extension in \eqref{eq:parent-fT} and using
Cauchy--Schwarz on the \(O(r^{-1/2})\)-fraction of normalized circle measure
meeting \(2\tau\),
\[
 \norm{f_T}_\infty
 \lesssim r^{-1/4}
 \norm{\wh{M_T\mu_1}}_{L^2(d\sigma_r)}
 +\RapDec(r).
\]
Therefore, with
\[
 A_T=\norm{\wh{M_T\mu_1}}_{L^2(d\sigma_r)},
\]
\begin{equation}\label{eq:parent-L6-square}
 \norm{\widetilde f_T}_6^2
 \lesssim r^{-2/3+\delta/3}A_T^2+\RapDec(r).
\end{equation}

Combining \eqref{eq:holder-local}--\eqref{eq:parent-L6-square} gives
\begin{equation}\label{eq:class-estimate}
 \int_{Y_M}\abs{\widetilde f_\lambda}^2\,d\mu_2
 \lesssim
 r^{-\frac{\alpha+1}{3}+\eps+C_\alpha\delta}
 \sum_{T\in\cW_\lambda}A_T^2
 +O_N(r^{-N}).
\end{equation}
There are \(O(\log r)\) packet-norm classes and \(O(\log r)\) nonzero
multiplicity classes.  If \(m(q)=0\), then \(100q\) misses \(3T\) for every
retained parent; the relative tail outside \(2.1T\), with the fixed margin
between \(2.1T\) and \(3T\), gives a rapidly decreasing contribution.
Cauchy--Schwarz over the class sums and absorption of logarithms into
\(r^\eps\) yield
\begin{equation}\label{eq:all-parents}
 \int_{E_2}\abs{\mu_{1,\mathrm{good}}*\wh{\sigma_r}}^2\,d\mu_2
 \lesssim
 r^{-\frac{\alpha+1}{3}+\eps+C_\alpha\delta}
 \sum_{R_j\sim r}\sum_\tau\sum_{T\in\mathbb T_{j,\tau}}A_T^2
 +O_N(r^{-N}).
\end{equation}

\begin{lemma}[Approximate orthogonality]
\label{lem:orthogonality}
There is a rapidly decreasing weight \(\psi_r\), comparable to one on
\(\{\xi:\abs{\abs{\xi}-r}\le1\}\), such that
\begin{equation}\label{eq:orthogonality-final}
 \sum_{R_j\sim r}\sum_\tau\sum_{T\in\mathbb T_{j,\tau}}
 \norm{\wh{M_T\mu_1}}_{L^2(d\sigma_r)}^2
 \lesssim r^{-1}
 \int_{\R^2}\abs{\wh\mu_1(\xi)}^2\psi_r(\xi)\,d\xi.
\end{equation}
\end{lemma}

\begin{proof}
This is the estimate proved in \cite[(5.3), pp.~35--37]{GIOW}.  Its proof
uses local constancy at unit frequency scale, enlarged cutoffs, Plancherel,
and bounded overlap; it does not use the disputed tube multiplicity.
\end{proof}

Substitute \eqref{eq:orthogonality-final} into
\eqref{eq:all-parents} and choose \(\delta\) so that
\(C_\alpha\delta<\eps\).  This proves \eqref{eq:main-frequency}, after
renaming the small exponent.  The range \(r\le10R_0\) is handled by the
elementary low-frequency estimate in \cite[Section~5]{GIOW}.

Integrating in \(r\) and using the Fourier representation of energy gives
\[
 \int_{E_2}\norm{d_x{}_*(\mu_{1,\mathrm{good}})}_{L^2}^2\,d\mu_2(x)
 \lesssim I_\gamma(\mu_1)+1,
 \qquad
 \gamma=2-\frac{\alpha+1}{3}+\eps.
\]
Since \(I_\gamma(\mu_1)<\infty\) whenever \(\gamma<\alpha\), the good part
closes for every \(\alpha>5/4\).  Together with
\cref{prop:bad-L1}, this proves the smoothed theorem.
\section{Uniform regularization and return to the original set}\label{sec:regularization}

The selected good part is initially most naturally a distribution, and the low-frequency term has a Schwartz tail. We give a uniform limiting argument that avoids treating it as a compactly supported finite measure without proof.

Choose a nonnegative $\varrho\in C_c^\infty(B(0,1))$ with integral one and set
\[
 \mu_i^\nu=\mu_i*\varrho_\nu,
 \qquad \varrho_\nu(x)=\nu^{-2}\varrho(x/\nu).
\]
For small $\nu$, the supports remain separated by a fixed positive distance, and every vector from $\supp\mu_1^\nu$ to $\supp\mu_2^\nu$ lies in a fixed slightly larger arc containing $\Omega$.

\begin{lemma}[Uniform Frostman bound]\label{lem:uniform-frostman}
If $\mu_i(B(x,t))\le C_\mu t^\alpha$ for $0<t\le1$, then
\[
 \mu_i^\nu(B(x,t))\le C'_\mu t^\alpha
\]
for all $0<t\le1$ and all sufficiently small $\nu$, with $C'_\mu$ independent of $\nu$.
\end{lemma}

\begin{proof}
If $t\ge\nu$, then
$\mu_i^\nu(B(x,t))\le\mu_i(B(x,t+\nu))\lesssim t^\alpha$.
If $t<\nu$, then
\[
 \mu_i^\nu(x)\lesssim \nu^{-2}\mu_i(B(x,\nu))\lesssim\nu^{\alpha-2},
\]
so
\[
 \mu_i^\nu(B(x,t))\lesssim t^2\nu^{\alpha-2}
 =t^\alpha(t/\nu)^{2-\alpha}\le t^\alpha.
\]
We may choose $\alpha<2$ if $\dim_{\mathrm H}E=2$.
\end{proof}

For fixed $\nu$, all selected packet sums converge as smooth functions. Run the bad-part estimate and the corrected good-part estimate uniformly for $(\mu_1^\nu,\mu_2^\nu)$. By \cref{lem:uniform-frostman}, all constants are independent of $\nu$. Thus there are sets $G_\nu$ with
$\mu_2^\nu(G_\nu)\ge1-10^{-3}$ such that
\begin{equation}\label{eq:uniform-L1}
 \norm{d_x{}_*(\mu_1^\nu)-d_x{}_*(\mu_{1,\mathrm{good}}^\nu)}_{L^1}
 \le10^{-3}\qquad(x\in G_\nu),
\end{equation}
and
\begin{equation}\label{eq:uniform-L2}
 \int\norm{d_x{}_*(\mu_{1,\mathrm{good}}^\nu)}_{L^2}^2\,d\mu_2^\nu(x)
 \le C
\end{equation}
with $C$ independent of $\nu$.

By Markov's inequality, choose
$x_\nu\in G_\nu\cap\supp\mu_2^\nu$ such that
\[
 \norm{d_{x_\nu}{}_*(\mu_{1,\mathrm{good}}^\nu)}_2\le2C^{1/2}=:C_0.
\]
Let
\[
 D_\nu=\{\abs{x_\nu-y}:y\in\supp\mu_1^\nu\}.
\]
The true pushforward is a probability measure supported in $D_\nu$. Hence \eqref{eq:uniform-L1} gives
\[
 \int_{D_\nu}\abs{d_{x_\nu}{}_*(\mu_{1,\mathrm{good}}^\nu)}
 \ge1-2\cdot10^{-3}.
\]
Cauchy--Schwarz yields
\begin{equation}\label{eq:Dnu-lower}
 \abs{D_\nu}\ge
 \left(\frac{1-2\cdot10^{-3}}{C_0}\right)^2=:c_0>0.
\end{equation}

Choose $\nu_k\downarrow0$ so that $x_{\nu_k}\to x\in E_2$. Since
$\supp\mu_1^{\nu_k}\subset\cN_{\nu_k}(E_1)$,
\[
 D_{\nu_k}\subset
 \cN_{t_k}(\Delta_x(E_1)),
 \qquad t_k=\abs{x_{\nu_k}-x}+\nu_k\to0.
\]
Pass to a subsequence with $t_k\downarrow0$. Compactness of $\Delta_x(E_1)$, \eqref{eq:Dnu-lower}, and continuity from above of Lebesgue measure give
\[
 \abs{\Delta_x(E_1)}
 =\lim_{k\to\infty}\abs{\cN_{t_k}(\Delta_x(E_1))}
 \ge c_0.
\]
This completes the proof of \cref{thm:main}.

\section{Conclusion}

The literal arbitrary-packet statement printed in the original
refined-decoupling argument is false under its stated support and
multiplicity conventions.  The collar construction makes the failure
quantitative.  The repair is to distinguish the tube supporting a packet
from the strictly larger tube used in the multiplicity count, and to retain
that margin through the induction on scales.

Once this enlargement-stable theorem is proved, the Falconer application
does not require a new canonical packetization.  Exact frequency truncation
places each original parent packet in the corrected theorem, while an
absolute small-packet cutoff converts rapid absolute tails into the required
relative tails.  The enlarged incidence count and the neighborhood forced
by local constancy then cancel with the same powers as in the intended
argument, leaving the exponent \(-(\alpha+1)/3\) and the threshold
\(\alpha>5/4\).

The theorem itself remains consistent with the later Riemannian proof of
Iosevich--Liu--Xi.  Carbery--Li--Pang--Yung independently provide a later
canonical activity-tube framework.  The contribution of the present paper
is the explicit counterexample and a complete, direct correction of the
original planar refined-decoupling proof route.

\section*{Statements and declarations}

\noindent\textbf{Funding.}
The authors received no external funding for this work.

\smallskip
\noindent\textbf{Competing interests.}
The authors declare no financial or non-financial competing interests
related to this work.

\smallskip
\noindent\textbf{Data availability.}
No datasets were generated or analyzed for this article.  The LaTeX source
is supplied with the submission.

\smallskip
\noindent\textbf{Author contributions.}
Shalender Singh conceived the central correction, developed the theoretical
argument, and drafted the manuscript.  Vishnu Priya Singh Parmar contributed
to conceptual refinement, mathematical analysis, and manuscript revision.
Both authors approved the submitted version and accept responsibility for
its content.
\appendix
\section{Proof of the enlargement-stable refined-decoupling estimate}
\label{app:refined-decoupling}

This appendix supplies the proof of \cref{thm:enlargement-stable}.  It is the
induction on scales underlying \cite[Theorem~4.2]{GIOW}, with two changes:
the tubes have the dimensionally consistent length-$R$ normalization, and a
packet is supported on an $a$-dilate while multiplicity is counted using a
strictly larger $b$-dilate.  The latter margin is tracked through every
rescaling.  No refined-decoupling theorem is used as an input.  A later canonical
wave-packet treatment for well-curved curves, including a full proof of the
canonical refined estimate, appears in Carbery--Li--Pang--Yung
\cite{CLPY}.  The present appendix proves the distinct enlargement-stable
\emph{arbitrary-packet} form needed for the direct Falconer-parent
application below.

\subsection{Normalized phases and the weighted decoupling input}
\label{app:normalized-family}

The circle is covered by finitely many fixed graph charts.  After fixed affine
normalizations, all chart phases and all their parabolic rescalings belong to
a family $\mathfrak G$ of smooth functions on $[-2,2]$ satisfying
\begin{equation}\label{eq:normalized-phase}
 g(0)=g'(0)=0,\qquad |g''(u)-2|\le\frac1{10},\qquad
 \|g\|_{C^m([-2,2])}\le C_m\quad(m\ge3),
\end{equation}
where the constants $C_m$ are fixed once and for all.  If $x_0\in[-1,1]$
and $R\ge16$, define
\begin{equation}\label{eq:phase-rescale}
 g_{x_0,R}(u)
 :=R^{1/2}\bigl[g(x_0+R^{-1/4}u)-g(x_0)
                  -R^{-1/4}g'(x_0)u\bigr].
\end{equation}
Then
\begin{align*}
 g_{x_0,R}''(u)&=g''(x_0+R^{-1/4}u),\\
 g_{x_0,R}^{(m)}(u)&=R^{(2-m)/4}g^{(m)}(x_0+R^{-1/4}u),
 \qquad m\ge3.
\end{align*}
so the rescaled phase remains in the same normalized family, after restricting
to the relevant fixed interval.

We use only the following weighted $\ell^2$ decoupling input.  It is the
normalized form of \cite[Theorem~4.1]{GIOW}, itself a consequence of the
Bourgain--Demeter theorem \cite{BDDecoupling}.

\begin{theorem}[Weighted $\ell^2$ decoupling input]
\label{thm:appendix-l2}
Fix $2\le p\le6$, $\epsilon'>0$, and $N_1\ge1$.  There is
$D=D(\epsilon',p,N_1,\{C_m\})$ such that the following holds for every phase
in the normalized rescaling family.  For $0<t\le1$, let $\mathcal J_t$ be a
bounded-overlap cover of $[-1,1]$ by intervals $J$ of length $2t^{1/2}$ and
put
\[
 \tau_J(t)=\{\xi:\xi_1\in J,\ |\xi_2-g(\xi_1)|\le t\}.
\]
If $F=\sum_{J\in\mathcal J_t}F_J$ and
$\supp\widehat F_J\subset\tau_J(t)$, then for every square $Q$ of side
$t^{-1}$,
\begin{equation}\label{eq:appendix-l2}
 \|F\|_{L^p(Q)}
 \le D t^{-\epsilon'}
 \left(\sum_{J\in\mathcal J_t}
       \|F_J\|_{L^p(w_Q)}^2\right)^{1/2},
\end{equation}
where $w_Q(x)=(1+t\dist(x,Q))^{-N_1}$.
\end{theorem}

The uniformity asserted here is the quantitative local form already
needed when \cite[Theorem~4.1]{GIOW} is applied after parabolic rescaling.
The Bourgain--Demeter induction depends only on the lower and upper curvature
bounds and on finitely many smooth seminorm bounds at the fixed decay order;
those quantities are bounded uniformly by \eqref{eq:normalized-phase} and
\eqref{eq:phase-rescale}.  Equivalently, one may take
\cref{thm:appendix-l2} as the weighted version of
\cite[Theorem~4.1]{GIOW}, with its implicit constant tracked on this single
normalized rescaling family.  No refined estimate, and no theorem for an
arbitrary class of phases, is being imported here.

\subsection{Configurations and exact rescaling}

Fix $g\in\mathfrak G$, a scale $R\ge2$, and
$\sigma_0\in[1/2,1]$.  Let $I$ range over an arbitrary translate of the
lattice intervals of length $R^{-1/2}$ whose centres have spacing
$\sigma_0R^{-1/2}$; this cover has overlap at most four.  Write $c_I$ for
the centre of $I$.  For $\Lambda\ge1$ put
\begin{equation}\label{eq:appendix-cap}
 \theta_I^\Lambda
 =\{\xi:\xi_1\in32I,\ |\xi_2-g(\xi_1)|\le4\Lambda R^{-1}\}.
\end{equation}
Let
\[
 \mathcal B_R=[-25R,25R]\times[-R,R]
\]
and, for $v\in\mathbb R$ and $s\ge1$, define
\begin{equation}\label{eq:appendix-strip}
 sT_{I,v}
 =\{x\in\mathcal B_R:
 |x_1+g'(c_I)x_2-v|\le s\Lambda R^{1/2}\}.
\end{equation}
A family has bounded overlap by label if, for each $I$, every point belongs
to at most four of the strips $T_{I,v}$ in the family.

For $1\le a<b\le20$ and $N_0=100$, call $F_T$ an $(a,N_0)$-packet on
$T=T_{I,v}$ if
\begin{equation}\label{eq:appendix-packet}
 \supp\widehat F_T\subset\theta_I^\Lambda,
 \qquad
 \|F_T\|_{L^p(\mathbb R^2\setminus aT)}
 \le R^{-N_0}\|F_T\|_p.
\end{equation}
Let $\mathfrak A(R,\Lambda,a,b)$ be the supremum, over all
$g\in\mathfrak G$, all $\sigma_0\in[1/2,1]$, and all allowed lattice
translates, of the least constant for which
\begin{equation}\label{eq:appendix-Adef}
 \left\|\sum_{T\in\mathcal W}F_T\right\|_{L^p(Y)}
 \le \mathfrak A(R,\Lambda,a,b)
 \left(\frac{M}{W}\right)^{\frac12-\frac1p}
 \left(\sum_{T\in\mathcal W}\|F_T\|_p^2\right)^{1/2}
\end{equation}
holds for the objects defined by that phase whenever
$W=|\mathcal W|$, the packet norms are comparable within a factor two,
$\mathcal W$ has bounded overlap by label, and $Y$ is a union of
$R^{1/2}$-squares satisfying
\begin{equation}\label{eq:appendix-multiplicity}
 \#\{T\in\mathcal W:bT\cap Q\ne\varnothing\}\le M
 \quad(Q\subset Y).
\end{equation}

For the inductive range take $R\ge2^{22}$; the finitely many smaller
scales are covered by \cref{lem:appendix-trivial} and absorbed into the final
constant.  Cover $[-1,1]$ by intervals
\[
 J=J_{x_0}=[x_0-R^{-1/4},x_0+R^{-1/4}],
 \qquad x_0\in\tfrac12R^{-1/4}\mathbb Z,
\]
of overlap at most four.  For each label interval $I$, choose $x_0$ with
$|x_0-c_I|\le\tfrac14R^{-1/4}$.  Since
\[
 16R^{-1/2}+\tfrac14R^{-1/4}\le R^{-1/4}
 \qquad(R\ge2^{22}),
\]
we have $32I\subset J_{x_0}$; fix one such interval and denote it by
$J(I)$.
The inclusion
\begin{equation}\label{eq:cap-in-tau}
 \theta_I^\Lambda
 \subset\tau_{J(I)}(R^{-1/2})
\end{equation}
holds whenever $\Lambda\le R^{1/2}/4$.
For $u\in2R^{3/4}\mathbb Z$ define the cylinder
\[
 \Box_{J,u}
 =\{x:|x_1+g'(x_0)x_2-u|\le25R^{3/4},\ |x_2|\le R\}.
\]
Put $R'=R^{1/2}$ and define
\begin{align}
 \Phi(x)
 &=\bigl(R^{-1/4}(x_1+g'(x_0)x_2-u),R^{-1/2}x_2\bigr),
 \label{eq:appendix-Phi}\\
 \Psi(\xi)
 &=\bigl(R^{1/4}(\xi_1-x_0),
 R^{1/2}(\xi_2-g(x_0)-g'(x_0)(\xi_1-x_0))\bigr).
 \label{eq:appendix-Psi}
\end{align}

\begin{lemma}[Exact parabolic rescaling]
\label{lem:appendix-rescaling}
Let $\widetilde g=g_{x_0,R}$.  Then:
\begin{enumerate}[label=\textup{(\alph*)},leftmargin=2em]
\item $\Phi(\Box_{J,u})=\mathcal B_{R'}$.  The inverse images of the
$R'^{1/2}$-squares are parallelograms of transverse width $R^{1/2}$ and
longitudinal length $R^{3/4}$.
\item If $I\subset J$ and $I'=R^{1/4}(I-x_0)$, then
\[
 \Psi(\theta_I^\Lambda)=(\theta_{I'}')^\Lambda,
 \qquad
 \Phi(sT_{I,v}\cap\Box_{J,u})=sT'_{I',v'},
 \quad v'=R^{-1/4}(v-u),
\]
where the primed objects are defined at scale $R'$ for $\widetilde g$ with
the same $\Lambda$.  The transformed intervals form another allowed translate of the
lower-scale lattice cover, with the same spacing parameter $\sigma_0$.
\item If
$H^\Phi(x')=(e^{-2\pi i(x_0,g(x_0))\cdot x}H(x))|_{x=\Phi^{-1}(x')}$,
then
\begin{align*}
 \supp\widehat{H^\Phi}&\subset\Psi(\supp\widehat H),\\
 \|H^\Phi\|_{L^p(\Phi(E))}
 &=|\det\Phi|^{1/p}\|H\|_{L^p(E)},
 \qquad |\det\Phi|=R^{-3/4}.
\end{align*}
\end{enumerate}
\end{lemma}

\begin{proof}
The determinant and the image of the cylinder follow directly from
\eqref{eq:appendix-Phi}.  The linear part of $\Psi$ is the inverse transpose
of the linear part of $\Phi$.  Moreover
\[
 R^{1/2}\{\xi_2-g(x_0)-g'(x_0)(\xi_1-x_0)\}
 -\widetilde g(R^{1/4}(\xi_1-x_0))
 =R^{1/2}(\xi_2-g(\xi_1)),
\]
which proves the cap identity.  Finally,
\[
 x_1+g'(c_I)x_2-v
 =R^{1/4}\{x_1'+\widetilde g'(c_{I'}')x_2'-v'\},
\]
and $\Lambda R^{1/2}=R^{1/4}\Lambda R'^{1/2}$, proving the strip identity.
The norm and Fourier-support formulas are the standard affine change of
variables.
\end{proof}

\subsection{The multiplicity transfer}

Assign $T=T_{I,v}$ to the cylinder $\Box(T)=\Box_{J(I),u(v)}$, where
$u(v)\in2R^{3/4}\mathbb Z$ is nearest to $v$.  If
$\Lambda\le R^{1/4}/100$ and $a\le10$, then
\begin{equation}\label{eq:aT-in-box}
 (a+1)T\subset\Box(T).
\end{equation}
Indeed, in the coordinates
$y_1=x_1+g'(x_0)x_2-u$, $y_2=x_2$, the three contributions from the strip
width, the change of slope, and $|v-u|$ are respectively
$O(\Lambda R^{1/2})$, $O(R^{3/4})$, and $O(R^{3/4})$.
Let $\mathcal W_\Box=\{T:\Box(T)=\Box\}$ and
$F_\Box=\sum_{T\in\mathcal W_\Box}F_T$.
For a rescaling parallelogram $C\subset\Box$, put
\[
 m_{\Box,b'}(C)=\#\{T\in\mathcal W_\Box:b'T\cap C\ne\varnothing\}.
\]
For an $R^{1/2}$-square $Q$, let
\[
 \mathcal N(Q)=\{Q'':\dist(Q,Q'')\le\Lambda^{1/4}R^{1/2}\};
 \qquad |\mathcal N(Q)|\lesssim\Lambda^{1/2}.
\]

\begin{lemma}[Transfer from a rescaling cell to the original square]
\label{lem:appendix-transfer}
Assume $R\ge2^{22}$, $64\le\Lambda\le R^{1/4}/100$, and
$C\cap Q''\ne\varnothing$ with $Q''\in\mathcal N(Q)$.  If
$b'T\cap C\ne\varnothing$, then the center of $Q$ lies in
\[
 (b'+12\Lambda^{-3/4})T.
\]
\end{lemma}

\begin{proof}
Write $\sigma=g'(c_I)-g'(x_0)$, so $|\sigma|\le2.1R^{-1/4}$ and
$T$ is described in the adapted coordinates by
$|y_1+\sigma y_2-(v-u)|\le\Lambda R^{1/2}$.  Choose
$z\in b'T\cap C$ and let $q$ be the center of $Q$.  Since $C$ has
$y_1$-width $R^{1/2}$ and $y_2$-length $R^{3/4}$, and since it meets a
square $Q''$ at distance at most $\Lambda^{1/4}R^{1/2}$ from $Q$,
\[
 |y_1(q)-y_1(z)|\le9\Lambda^{1/4}R^{1/2}.
\]
Moreover,
\begin{align*}
 |y_2(q)-y_2(z)|
 &\le R^{3/4}+\Lambda^{1/4}R^{1/2}+2.2R^{1/2}\\
 &=R^{3/4}\bigl[1+(\Lambda^{1/4}+2.2)R^{-1/4}\bigr]
 \le1.1R^{3/4}.
\end{align*}
The last inequality follows from
$\Lambda\le R^{1/4}/100$ and $R\ge2^{22}$.
Therefore
\begin{align*}
 |y_1(q)+\sigma y_2(q)-(v-u)|
 &\le b'\Lambda R^{1/2}
      +9\Lambda^{1/4}R^{1/2}
      +2.4R^{1/2}\\
 &\le(b'+12\Lambda^{-3/4})\Lambda R^{1/2}.
\end{align*}
This is the asserted containment.
\end{proof}

\subsection{The one-step recurrence}
\label{app:proof-refined}

A direct local estimate will be used at the terminal scale.

\begin{lemma}[Trivial bound]
\label{lem:appendix-trivial}
For $R\ge4$ and $b\le20$,
\[
 \mathfrak A(R,\Lambda,a,b)\lesssim R^{1/4}.
\]
If $\Lambda\ge R^{1/4}/100$, then
$\mathfrak A(R,\Lambda,a,b)\lesssim\Lambda$.
\end{lemma}

\begin{proof}
For an $R^{1/2}$-square $Q\subset Y$, at most $M$ supporting strips meet
$Q$, up to the $R^{-100}$ tails.  H\"older in the packet sum, followed by
summing the disjoint squares, gives
\[
 \|F\|_{L^p(Y)}
 \lesssim M^{1/2}
 \left(\frac{M}{W}\right)^{\frac12-\frac1p}
 \left(\sum_T\|F_T\|_p^2\right)^{1/2}.
\]
There are $O(R^{1/2})$ frequency labels, and bounded overlap by label gives
$M\lesssim R^{1/2}$.  Hence $M^{1/2}\lesssim R^{1/4}$; if
$\Lambda\ge R^{1/4}/100$, this is $O(\Lambda)$.
\end{proof}

\begin{proposition}[One induction step]
\label{prop:appendix-one-step}
Let $R\ge2^{22}$,
$R^{\vartheta/2}\le\Lambda\le R^{1/4}/100$, $\Lambda\ge64$, and
\[
 1\le a\le10,
 \qquad a+12\Lambda^{-3/4}\le b\le20.
\]
Take $N_1=\lceil2000/\vartheta\rceil$ in
\cref{thm:appendix-l2}, and let $D$ be the corresponding constant.  Then
\begin{equation}\label{eq:appendix-recurrence}
 \mathfrak A(R,\Lambda,a,b)
 \le C D(\log R)^4R^{\epsilon'/2}\Lambda^{1/2}
 \mathfrak A(R^{1/2},\Lambda,a,b-12\Lambda^{-3/4})+1.
\end{equation}
\end{proposition}

\begin{proof}
Fix a configuration and abbreviate the right-hand side of
\eqref{eq:appendix-Adef}, without $\mathfrak A$, by $\mathrm{RHS}$.  All
errors below are chosen smaller than $R^{-40}\mathrm{RHS}$; the polynomial
number of such errors accounts for the final additive constant.

First pigeonhole the $R^{1/2}$-squares so that
$\|F\|_{L^p(Q)}$ is comparable on the retained union $Y_0$ and
\begin{equation}\label{eq:appendix-square-pigeonhole}
 \|F\|_{L^p(Y)}^p\lesssim(\log R)\|F\|_{L^p(Y_0)}^p.
\end{equation}
For a large interval $J$, let
$F_{\tau_J}=\sum_{\Box=\Box_{J,u}}F_\Box$.  By
\eqref{eq:cap-in-tau}, its Fourier support lies in
$\tau_J(R^{-1/2})$.  Apply \cref{thm:appendix-l2} with $t=R^{-1/2}$ on a
square $Q\subset Y_0$ and set
\[
 \mathcal N^*(Q)
 =\{x:\dist(x,Q)\le\Lambda^{1/4}R^{1/2}\}.
\]
Since $\Lambda\ge R^{\vartheta/2}$ and
$N_1\ge2000/\vartheta$, the decoupling weight satisfies
\[
 w_Q(x)\le \Lambda^{-N_1/4}\le R^{-250}
 \qquad(x\notin\mathcal N^*(Q)).
\]
Moreover, \eqref{eq:aT-in-box} and the packet-tail hypothesis imply
\[
 \|F_\Box\|_{L^p(\mathbb R^2\setminus\Box)}
 \le \sum_{T\in\mathcal W_\Box}
       \|F_T\|_{L^p(\mathbb R^2\setminus aT)}
 \le R^{-100}\sum_{T\in\mathcal W_\Box}\|F_T\|_p.
\]
Consequently a cylinder $\Box_{J,u}$ that does not meet
$\mathcal N^*(Q)$ contributes only a negligible weighted tail.

For fixed $J$, use the adapted coordinate
$y_1=x_1+g'(x_0)x_2$.  The $y_1$-diameter of
$\mathcal N^*(Q)$ is at most
\[
 3.1(2\Lambda^{1/4}+1)R^{1/2}\le R^{3/4}.
\]
The cylinders with label $J$ have $y_1$-width $50R^{3/4}$ and their
centres are spaced by $2R^{3/4}$.  Hence at most $27$ of them meet
$\mathcal N^*(Q)$.  The triangle inequality followed by Cauchy--Schwarz
therefore gives
\[
 \|F_{\tau_J}\|_{L^p(w_Q)}
 \le 6\left(\sum_{\Box=\Box_{J,u}}
          \|F_\Box\|_{L^p(w_Q)}^2\right)^{1/2}
 +\text{negligible}.
\]

Inside $\mathcal N^*(Q)$, decompose into the grid squares
$Q''\in\mathcal N(Q)$.  Outside that collection the weight-tail estimate
above applies, and hence
\[
 \|F_\Box\|_{L^p(w_Q)}^p
 \lesssim\sum_{Q''\in\mathcal N(Q)}
       \|F_\Box\|_{L^p(Q'')}^p+\text{negligible}.
\]
Each $Q''$ meets at most eight of the rescaling parallelograms
$C\subset\Box$, and these cells have disjoint interiors.  Thus
\[
 \|F_\Box\|_{L^p(Q'')}^p
 \le\sum_{C\cap Q''\ne\varnothing}
       \|F_\Box\|_{L^p(Q''\cap C)}^p.
\]
Insert the last three displays into the weighted decoupling inequality and
use, since $p\ge2$,
$(\sum_i a_i^p)^{1/p}\le(\sum_i a_i^2)^{1/2}$.  We obtain
\begin{equation}\label{eq:appendix-triples}
 \|F\|_{L^p(Q)}
 \lesssim 6D R^{\epsilon'/2}
 \left(
  \sum_{(\Box,Q'',C)}
  \|F_\Box\|_{L^p(Q''\cap C)}^2
 \right)^{1/2},
\end{equation}
up to negligible terms.  Here $Q''\in\mathcal N(Q)$,
$C\cap Q''\ne\varnothing$, and the number of possible $Q''$ for fixed $Q$
is $O(\Lambda^{1/2})$.  This is the fully expanded weight-tail and
cell-decomposition step used below.

Put $b'=b-12\Lambda^{-3/4}$.  A non-negligible term in
\eqref{eq:appendix-triples} has $m_{\Box,b'}(C)\ge1$.  Pigeonhole the triples
according to dyadic values
\[
 m_{\Box,b'}(C)\sim M',\qquad |\mathcal W_\Box|\sim W',
\]
and then pigeonhole the squares according to the number
$n(Q)$ of retained triples.  There are only $O((\log R)^3)$ choices.  We
obtain a union $Y_2\subset Y_0$ on which $n(Q)\sim n$ and
\begin{equation}\label{eq:appendix-pigeonholed}
 \|F\|_{L^p(Y)}^p
 \lesssim (\log R)^4
 (D R^{\epsilon'/2})^p n^{p/2-1}
 \sum_Q\sum_{(\Box,Q'',C)\in\mathcal T(Q)}
 \|F_\Box\|_{L^p(Q''\cap C)}^p.
\end{equation}
A fixed triple occurs for at most $O(\Lambda^{1/2})$ choices of $Q$.  For
fixed $(\Box,C)$ the pieces $Q''\cap C$ are disjoint.  Therefore
\begin{equation}\label{eq:appendix-summed-triples}
 \|F\|_{L^p(Y)}^p
 \lesssim (\log R)^4(DR^{\epsilon'/2})^p
 n^{p/2-1}\Lambda^{1/2}
 \sum_{\Box:\,|\mathcal W_\Box|\sim W'}
 \|F_\Box\|_{L^p(Y_{\Box,M'})}^p,
\end{equation}
where $Y_{\Box,M'}$ is the union of cells $C$ with
$m_{\Box,b'}(C)\sim M'$.

Apply \cref{lem:appendix-rescaling} to a fixed cylinder.  The rescaled
packets form a configuration at scale $R^{1/2}$ with the same $\Lambda$,
support dilation $a$, counting dilation $b'$, multiplicity $O(M')$, and
packet count $|\mathcal W_\Box|$.  Hence, writing
$\mathfrak A'=\mathfrak A(R^{1/2},\Lambda,a,b')$,
\[
 \|F_\Box\|_{L^p(Y_{\Box,M'})}
 \lesssim \mathfrak A'
 \left(\frac{M'}{|\mathcal W_\Box|}\right)^{\frac12-\frac1p}
 \left(\sum_{T\in\mathcal W_\Box}\|F_T\|_p^2\right)^{1/2}.
\]
The families $\mathcal W_\Box$ are disjoint.  Since their sizes are
comparable to $W'$, summing the preceding inequality to the $p$th power gives
\begin{equation}\label{eq:appendix-induction-sum}
 \sum_{\Box:\,|\mathcal W_\Box|\sim W'}
 \|F_\Box\|_{L^p(Y_{\Box,M'})}^p
 \lesssim
 (\mathfrak A')^p
 \left(\frac{M'}{W}\right)^{p/2-1}
 \left(\sum_{T\in\mathcal W}\|F_T\|_p^2\right)^{p/2}.
\end{equation}

It remains to transfer the local multiplicity.  A fixed cylinder occurs in
at most $O(\Lambda^{1/2})$ triples attached to one $Q$.  Thus at least
$cn\Lambda^{-1/2}$ distinct cylinders occur.  Choose one retained triple in
each.  By \cref{lem:appendix-transfer}, all $\sim M'$ tubes counted by that
triple have their $b$-dilate meeting $Q$.  The cylinder families are
disjoint, and \eqref{eq:appendix-multiplicity} gives
\begin{equation}\label{eq:appendix-nM}
 nM'\lesssim\Lambda^{1/2}M.
\end{equation}
Substitute \eqref{eq:appendix-induction-sum} and
\eqref{eq:appendix-nM} into \eqref{eq:appendix-summed-triples}, take the
$p$th root, and use $2\le p\le6$.  All powers of $\Lambda$ are bounded by
$\Lambda^{1/2}$, and the result is \eqref{eq:appendix-recurrence}.
\end{proof}

\subsection{Iteration}

\begin{proof}[Proof of \cref{thm:enlargement-stable}]
Fix $\epsilon>0$ and take $\epsilon'>0$ sufficiently small.  Let
$R_k=R^{2^{-k}}$ and
\[
 b_k=b-Ck\Lambda^{-3/4},
\]
where $C$ is the absolute constant in \cref{lem:appendix-transfer}.  Stop at
the first $k_0$ for which $\Lambda\ge R_{k_0}^{1/4}/100$.  Since
$\Lambda\ge R^{\vartheta/2}$,
\[
 k_0\le C\log(1/\vartheta).
\]
For every $k<k_0$ we have
$R_k^{1/4}>100\Lambda\ge6400$, hence
$R_k>6400^4>2^{22}$; therefore the explicit lower-scale hypothesis in
\cref{prop:appendix-one-step} is satisfied at every inductive step.
For $R$ above a threshold depending only on $\vartheta$, the accumulated
drift $Ck_0\Lambda^{-3/4}$ is less than $1/4$, so $b_k-a\ge1/4$ throughout.
The finitely many smaller values of $R$ are absorbed into the implicit
constant.

Iterating \eqref{eq:appendix-recurrence} and using
\cref{lem:appendix-trivial} at $R_{k_0}$ gives
\[
 \mathfrak A(R,\Lambda,a,b)
 \lesssim_{\epsilon',\vartheta,p}
 \Lambda^{1+k_0/2}
 (\log R)^{Ck_0}
 R^{C\epsilon'}.
\]
Because $\Lambda\le R^{2\vartheta}$ and
$k_0=O(\log(1/\vartheta))$, the first factor is at most
$R^{C\vartheta\log(1/\vartheta)}$.  The logarithmic factor is absorbed by an
arbitrarily small further power of $R$.  Choose first $\vartheta_0(\epsilon)$
so that
$C\vartheta\log(1/\vartheta)<\epsilon/3$, and then choose
$\epsilon'$ so that $C\epsilon'<\epsilon/3$.  In the one-step estimate the
weighted decoupling constant is
\[
 D=D\!\left(\epsilon',p,
       N_1=\left\lceil\frac{2000}{\vartheta}\right\rceil,
       \{C_m\}\right).
\]
Thus, once $(\epsilon,\vartheta,p)$ are fixed, $D$ is a fixed constant; its
$\vartheta$-dependence is exactly part of the implicit constant allowed in
\eqref{eq:enlargement-stable}.  This proves that estimate with constant
$C(\epsilon,\vartheta,p)$.
\end{proof}

\section{Additional corrections to the published article}\label{app:additional-corrections}

These points are independent of the main refined-decoupling repair.

\subsection{The remainder in the general-curve identity}

In the proof of \cite[Lemma 7.5]{GIOW}, one obtains on a bounded $t$-interval
\[
 F(t)=F_1(t)+F_2(t)+O(A+B),
\]
where
\[
 A=\int_{\abs{\xi}\ge1}\abs{\wh f(\xi)}\abs{\xi}^{-3/2}\,d\xi,
 \qquad
 B=\int_{\abs{\xi}\le1}\abs{\wh f(\xi)}\abs{\xi}^{-1/2}\,d\xi.
\]
After squaring and integrating, the remainder is $O((A+B)^2)$, not $O(A+B)$. The subsequent displayed estimates already bound $A^2$ and $B^2$, so the stated Sobolev conclusion follows after this correction.

\subsection{\texorpdfstring{$\ell^p$}{lp} unit circles}

For $p>2$, the $\ell^p$ unit circle is $C^2$ and its curvature vanishes at the four axis points. For $1<p<2$, it is $C^1$ but not $C^2$ there; near $(1,0)$,
\[
 x=(1-\abs{y}^p)^{1/p}=1-\frac1p\abs{y}^p+O(\abs{y}^{2p}),
\]
and the curvature behaves like $\abs{y}^{p-2}$ and diverges. Thus the finite exceptional set should be described as the directions where the uniform smooth positive-curvature hypothesis fails, not uniformly as points of vanishing curvature. Localization away from those directions is unchanged.

\subsection{One pin for the upper Minkowski estimate}

An averaged estimate at each smoothing scale initially produces a pin depending on that scale. To select one pin, take $\delta_n=2^{-n^2}$. Apply Markov's inequality with exceptional measure at most $2^{-n-2}$. The union of the exceptional sets has measure at most $1/4$, while the loss $2^{n+2}=\delta_n^{-o(1)}$ does not change the exponent. Any point outside their union satisfies the required estimates for all $n$, and the sequence $\delta_n$ is sufficient for a lower bound on upper Minkowski dimension because that dimension is defined by a limsup.

\subsection{The infinite train-track paragraph}

The finite-scale train-track proposition in \cite[Section 6]{GIOW} is sufficient for the obstruction discussed there. The subsequent infinite-limit paragraph does not follow merely by writing intersections of independently defined train-track sets and renormalizing restrictions: nesting, nontriviality of the intersections, uniform mass allocation, and persistence of the finite-scale lower bounds all require proof. We therefore do not claim that paragraph as established. It should be omitted unless replaced by an explicit nested Moran construction. No part of the pinned $5/4$ theorem uses it.


\begin{thebibliography}{99}

\bibitem{BDDecoupling}
J.~Bourgain and C.~Demeter,
\emph{The proof of the $\ell^2$ decoupling conjecture},
Ann. of Math. (2) \textbf{182} (2015), no.~1, 351--389.

\bibitem{CLPY}
A.~Carbery, Z.~K.~Li, Y.~Pang and P.-L.~Yung,
\emph{A weighted formulation of refined decoupling and inequalities of
Mizohata--Takeuchi-type for the moment curve},
J. Geom. Anal. \textbf{36} (2026), article 48,
\url{https://doi.org/10.1007/s12220-025-02262-3}.

\bibitem{DemeterRefined}
C.~Demeter,
\emph{On the refined Strichartz estimates},
arXiv:2002.09525 (2020).

\bibitem{GIOW}
L.~Guth, A.~Iosevich, Y.~Ou and H.~Wang,
\emph{On Falconer's distance set problem in the plane},
Invent. Math. \textbf{219} (2020), no.~3, 779--830,
\url{https://doi.org/10.1007/s00222-019-00917-x}.

\bibitem{ILX}
A.~Iosevich, B.~Liu and Y.~Xi,
\emph{Microlocal decoupling inequalities and the distance problem on
Riemannian manifolds},
Amer. J. Math. \textbf{144} (2022), no.~6, 1601--1639,
\url{https://doi.org/10.1353/ajm.2022.0039}.

\bibitem{Liu}
B.~Liu,
\emph{An $L^2$-identity and pinned distance problem},
Geom. Funct. Anal. \textbf{29} (2019), no.~1, 283--294.

\bibitem{Orponen}
T.~Orponen,
\emph{On the dimension and smoothness of radial projections},
Anal. PDE \textbf{12} (2019), no.~5, 1273--1294.

\end{thebibliography}
\end{document}